\documentclass{article}%
\usepackage{amssymb}
\usepackage{graphicx}
\usepackage{amsmath}%
\usepackage{amsfonts}
\providecommand{\U}[1]{\protect\rule{.1in}{.1in}}
\newtheorem{theorem}{Theorem}

\newtheorem{definition}[theorem]{Definition}

\newtheorem{lemma}[theorem]{Lemma}
\newtheorem{notation}[theorem]{Notation}

\newtheorem{proposition}[theorem]{Proposition}
\newtheorem{remark}[theorem]{Remark}

\newenvironment{proof}[1][Proof]{\textbf{#1.} }{\ \rule{0.5em}{0.5em}}
\begin{document}

\title{Axiomatic Differential Geometry II-2\\-Its Developments-\\Chapter 2 ;Differential Forms}
\author{Hirokazu Nishimura\\Institute of Mathematics\\University of Tsukuba\\Tsukuba, Ibaraki, 305-8571, JAPAN}
\maketitle

\begin{abstract}
We refurbish our axiomatics of differential geometry introduced in
[Mathematics for Applications,, 1 (2012), 171-182]. Then the notion of
Euclideaness can naturally be formulated. The principal objective in this
paper is to present an adaptation of our theory of differential forms
developed in [International Journal of Pure and Applied Mathematics, 64
(2010), 85-102] to our present axiomatic framework.

\end{abstract}

\section{Introduction}

The principal objective in this paper is to replicate our treatment of
differential forms in \cite{nishi2} in the context of our axiomatics on
differential geometry in \cite{nishi3}. Trying to achieve this goal, we have
realized that our axiomatics there is somewhat fragile. Therefore, we were
forced to refurbish the axiomatics. The main improvement is that prolongations
of spaces with respect to Weil algebras can directly be generalized to those
with respect to finitely presented algebras. As is well known, the
prolongation of a space with respect to the Weil algebra $k\left[  X\right]
/\left(  X^{2}\right)  $\ (the Weil algebra corresponding to first-order
infinitesimals) is its tangent bundle. Similarly, the prolongation of a space
with respect to the polynomial algebra $k\left[  X_{1},...,X_{n}\right]  $,
which is not a Weil algebra but surely a finitely presented algebra,\ is
simply the exponentiation of the space by $\mathbb{R}^{n}$. Thus the secondary
objective in this paper is to improve our axiomatics, to which Section
\ref{s2} is devoted. In particular, the theorem established in \cite{nishi4}
that the tangent space is a module over $k$, which is external to the category
$\mathcal{K}$, is enhanced to the theorem that the tangent space is a module
over $\mathbb{R}$, which is an object in $\mathcal{K}$.

Section \ref{s3} is concerned with Euclidean modules. Our new axiomatics of
differential geometry enables us to formulate the notion of Euclideaness
properly, in which cartesian closedness and prolongations with respect to
polynomial algebras will play a crucial role. In orthodox differential
geometry and its extensions to infinite-dimensional differential geometry, we
first study the category of linear spaces of some kind and smooth mappings,
say, the category of Hilbert spaces, that of Banach spaces, that of
Fr\'{e}chet spaces, that of convenient vector spaces and so on. We then study
the category of manifolds, which are modeled locally after such linear spaces.
Our approach moves in the sheer opposite direction. We first establish the
general theory of microlinear spaces. The theory of Eucliean modules (i.e.,
its linear part) is obtained as a special case of this general theory.

Sections \ref{s4} and \ref{s5} are merely an adaptation of our treatment of
differential forms in \cite{nishi4} to our present axiomatic framework.
Section \ref{s4} is devoted to a unique characterization of differential
forms, which could be called \textit{the fundamental theorem on differential
forms}. The characterization and existence of exterior differentiation, which
will be discussed in Section \ref{s5}, is an easy consequence of this
fundamental theorem.

\section{\label{s2}Refurbishing our Axiomatics}

\subsection{\label{s2.1}The Refurbishment}

Let $k$\ be a commutative ring. We denote by $\mathbb{T}_{k}$ the algebraic
theory of $k$-algebras in the sense of Lawvere. We denote by $\mathbf{FP}%
\mathbb{T}_{k}$ the category of finitely presented $k$-algebras. It is well
known that Weil algebras over $k$\ are finitely presented $k$-algebras. We
denote by $\mathbf{Weil}_{k}$\ the category of Weil $k$-algebras, which is
well known to be left exact. In particular, its terminal object is
$k$\ itself. A finitely presented $k$-algebra $A$\ is called \textit{pointed
}if it has a unique maximal ideal $\mathfrak{m}$\ such that the composition of
the canonical morphism
\[
k\rightarrow A
\]
and the canonical projection
\[
A\rightarrow A/\mathfrak{m}%
\]
is an isomorphism. We denote by $\mathbf{PFP}\mathbb{T}_{k}$ the category of
pointed finitely presented $k$-algebras. Not only Weil $k$-algebras but also
polynomial $k$-algebras $k[X_{1},...,X_{n}]$\ lie in $\mathbf{PFP}%
\mathbb{T}_{k}$. Given a left exact category $\mathcal{K}$\ and a $k$-algebra
object $\mathbb{R}$\ in $\mathcal{K}$, there is a canonical functor
$\mathbb{R}\underline{\otimes}\cdot$\ (denoted by $\mathbb{R\otimes\cdot}$\ in
\cite{kock}) from the category $\mathbf{Weil}_{k}$ to the category of
$k$-algebra objects and their homomorphisms in $\mathcal{K}$.

\begin{definition}
\label{d2.1.1}(\underline{DG-category}) The present refinement of our original
axiomatics in \cite{nishi3}\ is that we allow not only Weil prolongations but
also \textit{finitely presented} prolongations. Therefore, given a finitely
presented $k$-algebra $A$, we are endowed with a left exact functor
$\mathbf{T}^{A}:\mathcal{K}\rightarrow\mathcal{K}$ preserving cartesian closed
structures in the sense that we have
\begin{equation}
\mathbf{T}^{A}\left(  X^{Y}\right)  =\left(  \mathbf{T}^{A}X\right)
^{Y}\label{2.1.0}%
\end{equation}
for any objects $X$\ and $Y$\ in $\mathcal{K}$. For any freely generated
$k$-algebra $A=k[X_{1},...,X_{n}]$ over $n$ generaters $X_{1},...,X_{n}$,
$\mathbf{T}^{A}=\mathbf{T}^{k[X_{1},...,X_{n}]}$ is required to be simply the
exponentiation by $\mathbb{R}^{n}$, so that we have
\begin{equation}
\mathbf{T}^{k[X_{1},...,X_{n}]}X=X^{\mathbb{R}^{n}}\label{2.1.1}%
\end{equation}
for any object $X$\ in $\mathcal{K}$. In particular, when $n=0$, we have
\[
\mathbf{T}^{k}X=X
\]
Given a finitely presented $k$-algebra $A$, it is required that
\begin{equation}
\mathbf{T}^{A}\mathbb{R=R}\underline{\mathbb{\otimes}}A\label{2.1.2}%
\end{equation}
Given two finitely presented $k$-algebra $A$ and $B$, it is required that
\begin{equation}
\mathbf{T}^{B}\circ\mathbf{T}^{A}=\mathbf{T}^{A\otimes_{k}B}\label{2.1.3}%
\end{equation}
Given a morphism $\varphi:A\rightarrow B$ in $\mathbf{PFP}\mathbb{T}_{k}$, we
have a natural transformation
\[
\alpha_{\varphi}:\mathbf{T}^{A}\Rightarrow\mathbf{T}^{B}%
\]
which respects cartesian closed structures, so that we have
\begin{equation}
\alpha_{\varphi}\left(  X^{Y}\right)  =\left(  \alpha_{\varphi}\left(
X\right)  \right)  ^{Y}\label{2.1.4}%
\end{equation}
for any objects $X$\ and $Y$\ in $\mathcal{K}$. It is also required to
satisfy
\begin{align}
\alpha_{\varphi}\left(  \mathbf{T}^{C}X\right)   & =\mathbf{T}^{C}\left(
\alpha_{\varphi}\left(  X\right)  \right)  :\mathbf{T}^{A}\mathbf{T}%
^{C}X=\mathbf{T}^{C\otimes_{k}A}X=\mathbf{T}^{A\otimes_{k}C}X=\mathbf{T}%
^{C}\mathbf{T}^{A}X\nonumber\\
& \rightarrow\mathbf{T}^{C}\mathbf{T}^{B}X=\mathbf{T}^{B\otimes_{k}%
C}X=\mathbf{T}^{C\otimes_{k}B}X=\mathbf{T}^{B}\mathbf{T}^{C}X\label{2.1.8}%
\end{align}
for any object $C$\ in the category $\mathbf{PFP}\mathbb{T}_{k}$. Given two
morphisms $\varphi:A\rightarrow B$ and $\psi:B\rightarrow C$ in $\mathbf{PFP}%
\mathbb{T}_{k}$, it is required that
\begin{equation}
\alpha_{\psi}\circ\alpha_{\varphi}=\alpha_{\psi\circ\varphi}\label{2.1.5}%
\end{equation}
Given any identity morphism $\mathrm{id}_{A}:A\rightarrow A$ in $\mathbf{PFP}%
\mathbb{T}_{k}$, it is required that
\begin{equation}
\alpha_{\mathrm{id}_{A}}=\mathrm{id}_{\mathbf{T}^{A}}\label{2.1.6}%
\end{equation}
Given a morphism $\varphi:A\rightarrow B$ in $\mathbf{PFP}\mathbb{T}_{k}$, it
is required that
\begin{equation}
\alpha_{\varphi}\left(  \mathbb{R}\right)  =\mathbb{R\underline
{\mathbb{\otimes}}\varphi}\label{2.1.7}%
\end{equation}
Thus our new definition of a \underline{DG-category} is a quadruple
\[
\left(  \mathcal{K},\mathbb{R},\mathbf{T},\alpha\right)
\]
where

\begin{enumerate}
\item $\mathcal{K}$ is a category which is left exact and cartesian closed.

\item $\mathbb{R}$ is a commutative $k$-algebra object in $\mathcal{K}$.

\item Given an object $A$\ in $\mathbf{PFP}\mathbb{T}_{k}$, $\mathbf{T}%
^{A}:\mathcal{K}\rightarrow\mathcal{K}$\ is a left-exact and
cartesian-closed-structure-preserving functor.

\item Given a morphism $\varphi:A\rightarrow B$ in $\mathbf{PFP}\mathbb{T}%
_{k}$, $\alpha_{\varphi}:\mathbf{T}^{A}\Rightarrow\mathbf{T}^{B}$ is a natural transformation.

\item The quadruple $\left(  \mathcal{K},\mathbb{R},\mathbf{T},\alpha\right)
$\ is required to satisfy (\ref{2.1.1})-(\ref{2.1.7}) as axioms.
\end{enumerate}
\end{definition}

\begin{remark}
As in \cite{nishi3}\ we have a bifunctor
\[
\otimes:\mathcal{K}\times\mathbf{PFP}\mathbb{T}_{k}\rightarrow\mathcal{K}%
\]
with
\[
X\otimes A=\mathbf{T}^{A}X
\]
for any object $X$\ in $\mathcal{K}$\ and any object $A$\ in $\mathbf{PFP}%
\mathbb{T}_{k}$, and
\begin{align*}
& f\otimes\varphi\\
& =\alpha_{\varphi}\left(  Y\right)  \circ\mathbf{T}^{A}f\\
& =\mathbf{T}^{B}f\circ\alpha_{\varphi}(X)
\end{align*}
for any morphism
\[
f:X\rightarrow Y
\]
in $\mathcal{K}$\ and any morphism
\[
\varphi:A\rightarrow B
\]
in $\mathbf{PFP}\mathbb{T}_{k}$.
\end{remark}

\begin{remark}
Given an object $A$\ in $\mathbf{PFP}\mathbb{T}_{k}$ and an object $X$\ in
$\mathcal{K}$, we write
\[
\tau_{A}\left(  X\right)  :\mathbf{T}^{A}X\rightarrow X
\]
and
\[
\iota_{A}\left(  X\right)  :X\rightarrow\mathbf{T}^{A}X
\]
for
\[
\alpha_{A\rightarrow k}\left(  X\right)  :\mathbf{T}^{A}X\rightarrow
\mathbf{T}^{k}X=X
\]
and
\[
\alpha_{k\rightarrow A}\left(  X\right)  :X=\mathbf{T}^{k}X\rightarrow
\mathbf{T}^{A}X
\]
respectively, where $A\rightarrow k$\ and $k\rightarrow A$\ are the canonical
morphisms in $\mathbf{PFP}\mathbb{T}_{k}$.
\end{remark}

It is easy to see that

\begin{proposition}
\label{t2.1.1}Let $\left(  \mathcal{K},\mathbb{R},\mathbf{T},\alpha\right)  $
be a DG-category with the category $\mathcal{K}$\ being locally cartesian
closed and $M$ an object\ in $\mathcal{K}$. Then $\left(  \mathcal{K}%
/M,\mathbb{R}_{M},\mathbf{T}_{M},\alpha^{M}\right)  $ is a DG-category but for
conditions (\ref{2.1.0}) and (\ref{2.1.4}), where

\begin{enumerate}
\item $\mathcal{K}/M$ is the slice category.

\item $\mathbb{R}_{M}$ is the canonical projection
\[
\mathbb{R}\times M\rightarrow M
\]

\item Given an object
\[
\pi:E\rightarrow M
\]
in $\mathcal{K}$ and an object $A$\ in $\mathbf{PFP}\mathbb{T}_{k}$,
$\mathbf{T}_{M}^{A}\left(  \pi\right)  $ is defined to be
\begin{equation}
\overline{\mathbf{T}}_{M}^{A}\left(  \pi\right)  \rightarrow M\label{t2.1.1.1}%
\end{equation}
where $\overline{\mathbf{T}}_{M}^{A}\left(  \pi\right)  $\ is obtained as the
equalizer of
\[
\mathbf{T}^{A}\pi:\mathbf{T}^{A}E\rightarrow\mathbf{T}^{A}M
\]
and
\[
\mathbf{T}^{A}E\underrightarrow{\,\tau_{A}\left(  E\right)  }%
\,E\,\underrightarrow{\pi}\,M\,\underrightarrow{\iota_{A}\left(  M\right)
}\,\mathbf{T}^{A}M
\]
and (\ref{t2.1.1.1}) is
\[
\overline{\mathbf{T}}_{M}^{A}\left(  \pi\right)  \rightarrow\mathbf{T}%
^{A}E\underrightarrow{\,\tau_{A}\left(  E\right)  }\,E\,\underrightarrow{\pi
}\,M
\]

\item Let $\varphi:A\rightarrow B$\ be a morphism in $\mathbf{PFP}%
\mathbb{T}_{k}$. Since the diagrams
\[%
\begin{array}
[c]{ccc}%
\mathbf{T}^{A}E & \underrightarrow{\alpha_{\varphi}\left(  E\right)  } &
\mathbf{T}^{B}E\\
\mathbf{T}^{A}\pi\downarrow &  & \downarrow\mathbf{T}^{B}\pi\\
\mathbf{T}^{A}M & \overrightarrow{\alpha_{\varphi}\left(  M\right)  } &
\mathbf{T}^{B}M
\end{array}
\]
\[%
\begin{array}
[c]{ccc}%
\mathbf{T}^{A}E & \underrightarrow{\alpha_{\varphi}\left(  E\right)  } &
\mathbf{T}^{B}E\\
\tau_{A}\left(  E\right)  \searrow &  & \swarrow\tau_{B}\left(  E\right) \\
& E & \\
& \pi\downarrow & \\
& M & \\
\iota_{A}\left(  M\right)  \swarrow &  & \searrow\iota_{B}\left(  M\right) \\
\mathbf{T}^{A}M & \overrightarrow{\alpha_{\varphi}\left(  M\right)  } &
\mathbf{T}^{B}M
\end{array}
\]
commute, there is a unique morphism
\[
\alpha_{\varphi}^{M}\left(  \pi\right)  :\overline{\mathbf{T}}_{M}^{A}\left(
\pi\right)  \rightarrow\overline{\mathbf{T}}_{M}^{B}\left(  \pi\right)
\]
in $\mathcal{K}$ such that the diagram
\[%
\begin{array}
[c]{ccc}%
\overline{\mathbf{T}}_{M}^{A}\left(  \pi\right)  & \underrightarrow
{\alpha_{\varphi}^{M}\left(  \pi\right)  } & \overline{\mathbf{T}}_{M}%
^{B}\left(  \pi\right) \\
\downarrow &  & \downarrow\\
\mathbf{T}^{A}E & \overrightarrow{\alpha_{\varphi}\left(  E\right)  } &
\mathbf{T}^{B}E
\end{array}
\]
commutes.
\end{enumerate}
\end{proposition}

\begin{definition}
(Local DG-category) A DG-category $\left(  \mathcal{K},\mathbb{R}%
,\mathbf{T},\alpha\right)  $ is called a \underline{local DG-category} if
$\mathcal{K}$ is locally cartesian closed and $\left(  \mathcal{K}%
/M,\mathbb{R}_{M},\mathbf{T}_{M},\alpha^{M}\right)  $ is a DG-category for any
object $M$\ in $\mathcal{K}$.
\end{definition}

\begin{remark}
The notion of microlinearity and that of Weil exponentiability remain the same
as those in \cite{nishi3}.
\end{remark}

In the following, we will consider an arbitrarily chosen local DG-category
$\left(  \mathcal{K},\mathbb{R},\mathbf{T},\alpha\right)  $ with $M$\ being a
microlinear and Weil exponentiable object in $\mathcal{K}$.

\subsection{\label{s2.2}The Duality}

We have already explained the duality between the category of Weil algebras in
the real world and the category of infinitesimal objects in the imaginary
world. Namely, we have a contravariant functor $\mathcal{D}$\ from the
category of Weil algebras to the category of infinitesimal objects and a
contravariant functor $\mathcal{W}$\ from the category of infinitesimal
objects to the category of Weil algebras, both of which constitute a dual
equivalence between the two categories. By way of example, $\mathcal{D}%
_{k\left[  X\right]  /\left(  X^{2}\right)  }$\ is intended for
\[
D=\left\{  x\in\mathbb{R}\mid x^{2}=0\right\}
\]
while $\mathcal{D}_{k\left[  X,Y\right]  /\left(  X^{2},Y^{2},XY\right)  }$ is
intended for
\[
D(2)=\left\{  \left(  x,y\right)  \in D\times D\mid xy=0\right\}
\]
Therefore we have
\begin{align*}
\mathcal{W}_{D}  & =k\left[  X\right]  /\left(  X^{2}\right) \\
\mathcal{W}_{D\left(  2\right)  }  & =k\left[  X,Y\right]  /\left(
X^{2},Y^{2},XY\right)
\end{align*}
Similarly
\[
\mathcal{W}_{d\in D\mapsto\left(  d,0\right)  \in D(2)}%
\]
stands for the homomorphism of $k$-algebras from $k\left[  X,Y\right]
/\left(  X^{2},Y^{2},XY\right)  $ to $k\left[  X\right]  /\left(
X^{2}\right)  $ assigning the equivalence class of $X$\ in $k\left[  X\right]
/\left(  X^{2}\right)  $ to the equivalence class of $X\ $in $\left[
X,Y\right]  /\left(  X^{2},Y^{2},XY\right)  $\ and assigning the equivalence
class of $0$\ in $k\left[  X\right]  /\left(  X^{2}\right)  $ to the
equivalence class of $Y\ $in $\left[  X,Y\right]  /\left(  X^{2}%
,Y^{2},XY\right)  $.

We can extend contravariant functors $\mathcal{D}$\ and $\mathcal{W}$\ so as
to yield a dual equivalence between the category $\mathbf{PFP}\mathbb{T}_{k}$
and the category of (real or imaginary) carved spaces standing for genuinely
formal $\mathrm{Spec}_{\mathbb{R}}$. By way of example, we have
\begin{align*}
\mathcal{D}_{k\left[  Z_{1},Z_{2}\right]  }  & =\mathbb{R}^{2}\\
\mathcal{D}_{k\left[  Z_{1},Z_{2},X,Y\right]  /\left(  X^{2},Y^{2},XY\right)
}  & =\mathbb{R}^{2}\times D(2)
\end{align*}
while
\[
\mathcal{W}_{_{d\in D\mapsto d\in\mathbb{R}}}%
\]
stands for the canonical projection
\[
k[X]\rightarrow k[X]/(X^{2})
\]
and
\[
\mathcal{W}_{\left(  r_{1},r_{2},d_{1},d_{2}\right)  \in\mathbb{R}^{2}\times
D(2)\mapsto r_{1}d_{1}+r_{2}d_{2}\in D}%
\]
stands for the homomorphism of $k$-algebras from $k\left[  X\right]  /\left(
X^{2}\right)  $ to $k\left[  Z_{1},Z_{2},X,Y\right]  /\left(  X^{2}%
,Y^{2},XY\right)  $ assigning the equivalence class of $Z_{1}X+Z_{2}Y$ in
$k\left[  Z_{1},Z_{2},X,Y\right]  /\left(  X^{2},Y^{2},XY\right)  $ to the
equivalence class of $X$\ in $k\left[  X\right]  /\left(  X^{2}\right)  $.

\subsection{\label{s2.3}The Tangent Space}

\begin{definition}
\label{d2.3.1}(\underline{Scalar Multiplication}) The exponential transpose of
the scalar multiplication
\[
\mathbb{R}\times\left(  M\otimes\mathcal{W}_{D}\right)  \rightarrow
M\otimes\mathcal{W}_{D}%
\]
is
\begin{align*}
\mathrm{id}_{M}\otimes\mathcal{W}_{\left(  r,d\right)  \in\mathbb{R}\times
D\mapsto rd\in D}  & :M\otimes\mathcal{W}_{D}\rightarrow M\otimes
\mathcal{W}_{\mathbb{R}\times D}=M\otimes\left(  \mathcal{W}_{D}\otimes
_{k}\mathcal{W}_{\mathbb{R}}\right)  =\\
\left(  M\otimes\mathcal{W}_{D}\right)  \otimes\mathcal{W}_{\mathbb{R}}  &
=\left(  M\otimes\mathcal{W}_{D}\right)  ^{\mathbb{R}}%
\end{align*}

\end{definition}

Now we strengthen one of the main results of \cite{nishi4}\ into

\begin{theorem}
\label{t2.3.1}The canonical projection
\[
\tau_{\mathcal{W}_{D}}\left(  M\right)  :M\otimes\mathcal{W}_{D}\rightarrow M
\]
is an $\mathbb{R}_{M}$-module in the slice category $\mathcal{K}/M$, where
$\mathbb{R}_{M}$\ is the canonical projection
\[
\mathbb{R}\times M\rightarrow M
\]

\end{theorem}

\begin{proof}
Here we deal only with the statement that the scalar multiplication
distributes over the addition, for which we have to verify that the diagram
\[%
\begin{array}
[c]{ccc}%
\mathbb{R}\times\left(  M\otimes\mathcal{W}_{D(2)}\right)  & \rightarrow &
\mathbb{R}\times\left(  M\otimes\mathcal{W}_{D}\right) \\
\downarrow &  & \downarrow\\
M\otimes\mathcal{W}_{D(2)} & \rightarrow & M\otimes\mathcal{W}_{D}%
\end{array}
\]
commutes, where the horizontal arrows stand for addition, while the vertical
arrows correspond to scalar multiplication. This follows easily from the
commutativity of the diagram
\[%
\begin{array}
[c]{ccc}%
M\otimes\mathcal{W}_{D(2)} & \underrightarrow{\mathrm{id}_{M}\otimes
\mathcal{W}_{d\in D\mapsto\left(  d,d\right)  \in D\left(  2\right)  }} &
M\otimes\mathcal{W}_{D}\\
\downarrow &  & \downarrow\\
M\otimes\mathcal{W}_{\mathbb{R}\times D\left(  2\right)  } & \overrightarrow
{\mathrm{id}_{M}\otimes\mathcal{W}_{\left(  r,d\right)  \in\mathbb{R}\times
D\mapsto\left(  r,d,d\right)  \in\mathbb{R}\times D\left(  2\right)  }} &
M\otimes\mathcal{W}_{\mathbb{R}\times D}%
\end{array}
\]
where the left vertical arrow is
\[
\mathrm{id}_{M}\otimes\mathcal{W}_{\left(  r,d_{1},d_{2}\right)  \in
\mathbb{R}\times D\mapsto\left(  rd_{1},rd_{2}\right)  \in D}%
\]
while the right vertical arrow is
\[
\mathrm{id}_{M}\otimes\mathcal{W}_{\left(  r,d\right)  \in\mathbb{R}\times
D\mapsto rd\in D}%
\]

\end{proof}

\section{\label{s3}Euclidean Modules}

An $\mathbb{R}$-module in $\mathcal{K}$\ is an object $\mathbb{E}$\ in
$\mathcal{K}$\ endowed with a morphism
\[
\mathbb{+}_{\mathbb{E}}:\mathbb{E\times E\rightarrow E}\text{,}%
\]
intended for addition, and a morphism
\[
\cdot_{\mathbb{R},\mathbb{E}}:\mathbb{R}\times\mathbb{E\rightarrow E}\text{,}%
\]
intended for scalar multiplication, which are surely subject to the usual
axioms of an $\mathbb{R}$-module depicted diagrammatically. Equivalently, an
$\mathbb{R}$-module structure on an object $\mathbb{E}$\ in $\mathcal{K}$\ can
be given by a single morphism
\[
\varphi:\mathbb{R}\times\mathbb{E\times E\rightarrow E}%
\]
intended for the morphism
\[
\mathbb{R}\times\mathbb{E\times E\,}\underrightarrow{\cdot_{\mathbb{R}%
,\mathbb{E}}\times\mathrm{id}_{\mathbb{E}}}\,\mathbb{E\times E\,}%
\underrightarrow{\mathbb{\mathbb{+}_{\mathbb{E}}}}\mathbb{\,E}%
\]
which is surely subject to some axioms depicted diagrammatically.

\begin{definition}
\label{d3.1}(\underline{Euclidean $\mathbb{R}$-module}) An $\mathbb{R}%
$-module$\mathbb{E}$ \ is called \underline{Euclidean} provided that the
composition of the exponential transpose
\begin{equation}
\mathbb{E\times E\rightarrow E}^{\mathbb{R}}\label{d3.1.1}%
\end{equation}
of
\begin{equation}
\mathbb{R}\times\mathbb{E\times E=E\times R}\times\mathbb{E\,}\underrightarrow
{\mathrm{id}_{\mathbb{E}}\times\cdot_{\mathbb{R},\mathbb{E}}}\,\mathbb{E\times
E\,}\underrightarrow{\mathbb{+}_{\mathbb{E}}}\,\mathbb{E}\label{d3.1.2}%
\end{equation}
and
\begin{equation}
\alpha_{\mathcal{W}_{d\in D\mapsto d\in\mathbb{R}}}\left(  \mathbb{E}\right)
:\mathbb{E}^{\mathbb{R}}=\mathbb{E}\otimes\mathcal{W}_{\mathbb{R}}%
\rightarrow\mathbb{E}\otimes\mathcal{W}_{D}\label{d3.1.3}%
\end{equation}
in succession is an isomorphism.
\end{definition}

It should be obvious that

\begin{lemma}
\label{t3.1}The $\mathbb{R}$-module structure of $\mathbb{E}$\ naturally gives
rise to that of $\mathbb{E}^{X}$ for any object $X$ in $\mathcal{K}$ in the
sense that the exponential transpose
\[
\widetilde{\varphi}:\mathbb{E\times E\rightarrow E}^{\mathbb{R}}%
\]
of the $\mathbb{R}$-module structure
\[
\varphi:\mathbb{R}\times\mathbb{E\times E\rightarrow E}%
\]
on $\mathbb{E}$ induces a mapping
\[
\left(  \widetilde{\varphi}\right)  ^{X}:\mathbb{E}^{X}\mathbb{\times E}%
^{X}=\left(  \mathbb{E\times E}\right)  ^{X}\,\rightarrow\,\left(
\mathbb{E}^{\mathbb{R}}\right)  ^{X}=\left(  \mathbb{E}^{X}\right)
^{\mathbb{R}}\text{,}%
\]
which is the exponential transpose of the derived $\mathbb{R}$-module
structure
\[
\mathbb{R}\times\mathbb{E}^{X}\mathbb{\times E}^{X}\rightarrow\mathbb{E}^{X}%
\]
on $\mathbb{E}^{X}$.
\end{lemma}

\begin{proposition}
\label{t3.2}If $\mathbb{E}$ is a Euclidean $\mathbb{R}$-module, then so is
$\mathbb{E}^{X}$ for any object $X$ in $\mathcal{K}$.
\end{proposition}

\begin{proof}
We use the same notation as in Lemma \ref{t3.1}. We have
\begin{align*}
\alpha_{\mathcal{W}_{_{d\in D\mapsto d\in\mathbb{R}}}}\left(  \mathbb{E}%
^{X}\right)   & =\left(  \alpha_{\mathcal{W}_{_{d\in D\mapsto d\in\mathbb{R}}%
}}\left(  \mathbb{E}\right)  \right)  ^{X}:\left(  \mathbb{E}^{X}\right)
^{\mathbb{R}}=\left(  \mathbb{E}^{\mathbb{R}}\right)  ^{X}=\left(
\mathbb{E}\otimes\mathcal{W}_{\mathbb{R}}\right)  ^{X}\rightarrow\\
\left(  \mathbb{E}\otimes\mathcal{W}_{D}\right)  ^{X}  & =\mathbb{E}%
^{X}\otimes\mathcal{W}_{D}%
\end{align*}
Since
\[
\left(  \alpha_{\mathcal{W}_{_{d\in D\mapsto d\in\mathbb{R}}}}\left(
\mathbb{E}\right)  \right)  ^{X}\circ\left(  \widetilde{\varphi}\right)
^{X}=\left(  \alpha_{\mathcal{W}_{_{d\in D\mapsto d\in\mathbb{R}}}}\left(
\mathbb{E}\right)  \circ\widetilde{\varphi}\right)  ^{X}%
\]
we are sure that $\mathbb{E}^{X}$ is a Euclidean $\mathbb{R}$-module.
\end{proof}

It should be evident that

\begin{lemma}
\label{t3.3}The $\mathbb{R}$-module structure of $\mathbb{E}$\ naturally gives
rise to that of $\mathbb{E}\otimes W$ for any Weil algebra $W$ in the sense
that the exponential transpose
\[
\widetilde{\varphi}:\mathbb{E\times E\rightarrow E}^{\mathbb{R}}%
\]
of the $\mathbb{R}$-module structure
\[
\varphi:\mathbb{R}\times\mathbb{E\times E\rightarrow E}%
\]
on $\mathbb{E}$ induces a mapping
\[
\widetilde{\varphi}\otimes\mathrm{id}_{W}:(\mathbb{E}\otimes W)\mathbb{\times
}(\mathbb{E}\otimes W)=\left(  \mathbb{E\times E}\right)  \otimes
W\rightarrow\mathbb{E}^{\mathbb{R}}\otimes W=\left(  \mathbb{E}\otimes
W\right)  ^{\mathbb{R}}\text{,}%
\]
which is the exponential transpose of the derived $\mathbb{R}$-module
structure
\[
\mathbb{R}\times(\mathbb{E}\otimes W)\mathbb{\times}(\mathbb{E}\otimes
W)\rightarrow\mathbb{E}\otimes W
\]
on $\mathbb{E}\otimes W$.
\end{lemma}

\begin{proposition}
\label{t3.4}If $\mathbb{E}$ is a Euclidean $\mathbb{R}$-module, then so is
$\mathbb{E}\otimes W$ for any Weil algebra $W$.
\end{proposition}

\begin{proof}
We use the same notation as in Lemma \ref{t3.3}. We have
\begin{align*}
\alpha_{\mathcal{W}_{_{d\in D\mapsto d\in\mathbb{R}}}}\left(  \mathbb{E}%
\otimes W\right)   & =\alpha_{\mathcal{W}_{_{d\in D\mapsto d\in\mathbb{R}}}%
}\left(  \mathbb{E}\right)  \otimes\mathrm{id}_{W}:\left(  \mathbb{E}\otimes
W\right)  ^{\mathbb{R}}=\left(  \mathbb{E}\otimes W\right)  \otimes
\mathcal{W}_{\mathbb{R}}\\
& =\left(  \mathbb{E}\otimes\mathcal{W}_{\mathbb{R}}\right)  \otimes
W\rightarrow\left(  \mathbb{E}\otimes\mathcal{W}_{D}\right)  \otimes W=\left(
\mathbb{E}\otimes W\right)  \otimes\mathcal{W}_{D}%
\end{align*}
Since
\[
\left(  \alpha_{\mathcal{W}_{_{d\in D\mapsto d\in\mathbb{R}}}}\left(
\mathbb{E}\right)  \otimes\mathrm{id}_{W}\right)  \circ\left(  \widetilde
{\varphi}\otimes\mathrm{id}_{W}\right)  =\left(  \alpha_{\mathcal{W}_{_{d\in
D\mapsto d\in\mathbb{R}}}}\left(  \mathbb{E}\right)  \circ\widetilde{\varphi
}\right)  \otimes\mathrm{id}_{W}%
\]
we are sure that $\mathbb{E}\otimes W$ is a Euclidean $\mathbb{R}$-module.
\end{proof}

\begin{remark}
\label{r3.1}If $\mathbb{E}$ is an $\mathbb{R}$-module, then the first
projection
\[
\pi_{1}:\mathbb{E\times E\rightarrow E}%
\]
is naturally an $\mathbb{R}$-module in the slice category $\mathcal{K}%
/\mathbb{E}$.
\end{remark}

\begin{proposition}
\label{t3.5}If $\mathbb{E}$ is a Euclidean $\mathbb{R}$-module, then the
identification of $\mathbb{E}\otimes\mathcal{W}_{D}$ and $\mathbb{E\times E} $
in Definition \ref{d3.1} together with the commutative diagram
\[%
\begin{array}
[c]{ccc}%
\mathbb{E\times E} & = & \mathbb{E}\otimes\mathcal{W}_{D}\\
\pi_{1}\searrow &  & \swarrow\tau_{\mathcal{W}_{D}}\left(  \mathbb{E}\right)
\\
& \mathbb{E} &
\end{array}
\]
allows us to identify the $\mathbb{R}_{M}$-module structure in Theorem
\ref{t2.3.1}\ and that in Remark \ref{r3.1}.
\end{proposition}

\begin{proof}
\begin{enumerate}
\item First we deal with addition. We have
\begin{align}
\mathbb{E}\otimes W_{D(2)}  & =\left(  \mathbb{E}\otimes\mathcal{W}%
_{D}\right)  \times_{M}\left(  \mathbb{E}\otimes\mathcal{W}_{D}\right)
\nonumber\\
& =\left(  \underset{1}{\mathbb{E}}\mathbb{\times}\underset{2}{\mathbb{E}%
}\right)  \times_{M}\left(  \underset{3}{\mathbb{E}}\mathbb{\times}%
\underset{4}{\mathbb{E}}\right)  \nonumber\\
& =\underset{1,3}{\mathbb{E}}\mathbb{\times}\underset{2}{\mathbb{E}%
}\mathbb{\times}\underset{4}{\mathbb{E}}\label{3.5.1}%
\end{align}
where the numbers under $\mathbb{E}$ are given simply so as for the reader to
relate each occurrence of $\mathbb{E}$\ on the last line to the appropriate
occurrence of $\mathbb{E}$\ on the previous line. This isomorphism can be
realized by the composition of the exponential transpose
\begin{equation}
\mathbb{E\times E\times E\rightarrow E}^{\mathbb{R\times R}}\label{3.5.2}%
\end{equation}
of
\begin{align}
\underset{1}{\mathbb{R}}\times\underset{2}{\mathbb{R}}\times\underset
{3}{\mathbb{E}}\mathbb{\times}\underset{4}{\mathbb{E}}\mathbb{\times}%
\underset{5}{\mathbb{E}}  & =\underset{3}{\mathbb{E}}\mathbb{\times}%
\underset{1}{\mathbb{\mathbb{R}}}\mathbb{\times}\underset{4}%
{\mathbb{\mathbb{E}}}\mathbb{\times}\underset{2}{\mathbb{R}}\times\underset
{5}{\mathbb{E}}\mathbb{\,}\underrightarrow{\mathrm{id}_{\mathbb{E}}\times
\cdot_{\mathbb{R},\mathbb{E}}\times\cdot_{\mathbb{R},\mathbb{E}}%
}\,\mathbb{E\times E\times E}\nonumber\\
& \,\underrightarrow{\mathrm{id}_{\mathbb{E}}\times\mathbb{+}_{\mathbb{E}}%
}\,\mathbb{E\times E\,}\underrightarrow{\mathbb{+}_{\mathbb{E}}}%
\,\mathbb{E}\label{3.5.3}%
\end{align}
and
\begin{equation}
\alpha_{\mathcal{W}_{\left(  d_{1},d_{2}\right)  \in D\left(  2\right)
\mapsto\left(  d_{1},d_{2}\right)  \in\mathbb{R}\times\mathbb{R}}}\left(
\mathbb{E}\right)  :\mathbb{E}^{\mathbb{R}\times\mathbb{R}}=\mathbb{E}%
\otimes\mathcal{W}_{\mathbb{R}\times\mathbb{R}}\rightarrow\mathbb{E}%
\otimes\mathcal{W}_{D\left(  2\right)  }\label{3.5.4}%
\end{equation}
in succession, where the numbers under $\mathbb{R}$\ and $\mathbb{E}$\ are
intended for the reader to easily relate their occurrences on the first line
to those on the second line. Therefore, the commutativity of the diagrams
\[%
\begin{array}
[c]{ccc}%
\mathbb{E\times E\times E} & \rightarrow & \mathbb{E}^{\mathbb{R}%
\times\mathbb{R}}=\mathbb{E}\otimes\mathcal{W}_{\mathbb{R}\times\mathbb{R}}\\%
\begin{array}
[c]{c}%
\mathrm{id}_{\mathbb{E}}\times\mathbb{+}_{\mathbb{E}}%
\end{array}
\downarrow &  & \downarrow%
\begin{array}
[c]{c}%
\mathrm{id}_{\mathbb{E}}\otimes\mathcal{W}_{r\in\mathbb{R}\mapsto\left(
r,r\right)  \in\mathbb{R}\times\mathbb{R}}%
\end{array}
\\
\mathbb{E\times E} & \rightarrow & \mathbb{E}^{\mathbb{R}}=\mathbb{E}%
\otimes\mathcal{W}_{\mathbb{R}}%
\end{array}
\]%
\[%
\begin{array}
[c]{ccc}%
\mathbb{E}^{\mathbb{R}\times\mathbb{R}}=\mathbb{E}\otimes\mathcal{W}%
_{\mathbb{R}\times\mathbb{R}} & \rightarrow & \mathbb{E}\otimes\mathcal{W}%
_{D\left(  2\right)  }\\%
\begin{array}
[c]{c}%
\mathrm{id}_{\mathbb{E}}\otimes\mathcal{W}_{r\in\mathbb{R}\mapsto\left(
r,r\right)  \in\mathbb{R}\times\mathbb{R}}%
\end{array}
\downarrow &  & \downarrow%
\begin{array}
[c]{c}%
\mathrm{id}_{\mathbb{E}}\otimes\mathcal{W}_{d\in D\mapsto\left(  d,d\right)
\in D\left(  2\right)  }%
\end{array}
\\
\mathbb{E}^{\mathbb{R}}=\mathbb{E}\otimes\mathcal{W}_{\mathbb{R}} &
\rightarrow & \mathbb{E}\otimes\mathcal{W}_{D}%
\end{array}
\text{,}%
\]
with the morphism
\[
\mathbb{E\times E\times E\rightarrow E}^{\mathbb{R}\times\mathbb{R}}%
\]
being that in (\ref{3.5.2}), the morphism
\[
\mathbb{E\times E\rightarrow E}^{\mathbb{R}}%
\]
being that in (\ref{d3.1.1}), the moriphism%
\[
\mathbb{E}^{\mathbb{R}\times\mathbb{R}}=\mathbb{E}\otimes\mathcal{W}%
_{\mathbb{R}\times\mathbb{R}}\rightarrow\mathbb{E}\otimes\mathcal{W}_{D\left(
2\right)  }%
\]
being the morphism%
\[
\mathrm{id}_{\mathbb{E}}\otimes\mathcal{W}_{\left(  d_{1},d_{2}\right)  \in
D\left(  2\right)  \mapsto\left(  d_{1},d_{2}\right)  \in\mathbb{R}%
\times\mathbb{R}}%
\]
and the morphism%
\[
\mathbb{E}^{\mathbb{R}}=\mathbb{E}\otimes\mathcal{W}_{\mathbb{R}}%
\rightarrow\mathbb{E}\otimes\mathcal{W}_{D}%
\]
being the morphism%
\[
\mathrm{id}_{\mathbb{E}}\otimes\mathcal{W}_{_{d\in D\mapsto d\in\mathbb{R}}}%
\]
implies the commutativity of the diagram
\[%
\begin{array}
[c]{ccc}%
\mathbb{E\times E\times E} & = & \mathbb{E}\otimes\mathcal{W}_{D\left(
2\right)  }\\%
\begin{array}
[c]{c}%
\mathrm{id}_{\mathbb{E}}\times\mathbb{+}_{\mathbb{E}}%
\end{array}
\downarrow &  & \downarrow%
\begin{array}
[c]{c}%
\mathrm{id}_{\mathbb{E}}\otimes\mathcal{W}_{d\in D\mapsto\left(  d,d\right)
\in D\left(  2\right)  }%
\end{array}
\\
\mathbb{E\times E} & = & \mathbb{E}\otimes\mathcal{W}_{D}%
\end{array}
\]
This is no other than the gist of the desired statement.

\item Now we deal with scalar multiplication. The commutativity of the
diagrams
\[%
\begin{array}
[c]{ccc}%
\mathbb{E\times E} & \rightarrow & \mathbb{E}^{\mathbb{R}}=\mathbb{E}%
\otimes\mathcal{W}_{\mathbb{R}}\\
\downarrow &  & \downarrow\mathrm{id}_{\mathbb{E}}\otimes\mathcal{W}_{\left(
r_{1},r_{2}\right)  \in\mathbb{R\times R}\mapsto r_{1}r_{2}\in\mathbb{R}}\\
\left(  \mathbb{E\times E}\right)  ^{\mathbb{R}} & \rightarrow &
\mathbb{E}^{\mathbb{R\times R}}=\mathbb{E}\otimes\mathcal{W}_{\mathbb{R\times
R}}%
\end{array}
\]%
\[%
\begin{array}
[c]{ccc}%
\mathbb{E}^{\mathbb{R}}=\mathbb{E}\otimes\mathcal{W}_{\mathbb{R}} &
\rightarrow & \mathbb{E}\otimes\mathcal{W}_{D}\\
\mathrm{id}_{\mathbb{E}}\otimes\mathcal{W}_{\left(  r_{1},r_{2}\right)
\in\mathbb{R\times R}\mapsto r_{1}r_{2}\in\mathbb{R}}\downarrow &  &
\downarrow\mathrm{id}_{\mathbb{E}}\otimes\mathcal{W}_{\left(  d,r\right)  \in
D\times\mathbb{R}\mapsto dr\in D}\\
\mathbb{E}^{\mathbb{R\times R}}=\mathbb{E}\otimes\mathcal{W}_{\mathbb{R\times
R}} & \rightarrow & \mathbb{E}\otimes\mathcal{W}_{D\times\mathbb{R}}=\left(
\mathbb{E}\otimes\mathcal{W}_{D}\right)  ^{\mathbb{R}}%
\end{array}
\]
with the left vertical arrow in the first diagram
\[
\mathbb{E\times E\rightarrow}\left(  \mathbb{E\times E}\right)  ^{\mathbb{R}}%
\]
being the exponential transpose of
\begin{equation}
\mathbb{R\times E\times E=E\times R\times E\,}\underrightarrow{\mathrm{id}%
_{\mathbb{E}}\times\cdot_{\mathbb{R},\mathbb{E}}}\,\mathbb{E\times
E}\label{3.5.5}%
\end{equation}
the upper horizontal arrow
\[
\mathbb{E\times E\rightarrow E}^{\mathbb{R}}%
\]
in the first diagram being that in (\ref{d3.1.1}), the lower horizontal arrow
\[
\left(  \mathbb{E\times E}\right)  ^{\mathbb{R}}\rightarrow\mathbb{E}%
^{\mathbb{R\times R}}=\mathbb{E}\otimes\mathcal{W}_{\mathbb{R\times R}}%
\]
in the first diagram being that in (\ref{d3.1.1}) exponentiated by
$\mathbb{R}$, the upper horizontal arrow%
\[
\mathbb{E}^{\mathbb{R}}=\mathbb{E}\otimes\mathcal{W}_{\mathbb{R}}%
\rightarrow\mathbb{E}\otimes\mathcal{W}_{D}%
\]
in the second diagram being%
\[
\mathrm{id}_{\mathbb{E}}\otimes\mathcal{W}_{d\in D\mapsto d\in\mathbb{R}}%
\]
and the lower horizontal arrow%
\[
\mathbb{E}^{\mathbb{R\times R}}=\mathbb{E}\otimes\mathcal{W}_{\mathbb{R\times
R}}\rightarrow\mathbb{E}\otimes\mathcal{W}_{D\times\mathbb{R}}=\left(
\mathbb{E}\otimes\mathcal{W}_{D}\right)  ^{\mathbb{R}}%
\]
in the second diagram being%
\[
\mathrm{id}_{\mathbb{E}}\otimes\mathcal{W}_{\left(  d,r\right)  \in
D\times\mathbb{R}\mapsto\left(  d,r\right)  \in\mathbb{R\times R}}%
\]
implies the commutativity of the diagram
\[%
\begin{array}
[c]{ccc}%
\mathbb{E\times E} & = & \mathbb{E}\otimes\mathcal{W}_{D}\\
\downarrow &  & \downarrow\\
\left(  \mathbb{E\times E}\right)  ^{\mathbb{R}} & = & \left(  \mathbb{E}%
\otimes\mathcal{W}_{D}\right)  ^{\mathbb{R}}%
\end{array}
\]
which is the exponential transpose of the commutative diagram
\[%
\begin{array}
[c]{ccc}%
\mathbb{R\times E\times E} & = & \mathbb{R}\times\left(  \mathbb{E}%
\otimes\mathcal{W}_{D}\right)  \\
\downarrow &  & \downarrow\\
\mathbb{E\times E} & = & \mathbb{E}\otimes\mathcal{W}_{D}%
\end{array}
\]
with the left vertical arrow being that in (\ref{3.5.5}) and the right
vertical arrow being the scalar multiplication in Definition \ref{d2.3.1}.This
is no other than the gist of the desired statement.
\end{enumerate}
\end{proof}

It should be apparent that

\begin{lemma}
\label{t3.6}The diagram
\[
\mathcal{W}_{D(2)}\,\underrightarrow{\mathcal{W}_{\left(  d_{1},d_{2}\right)
\in D^{2}\mapsto\left(  d_{1},d_{1}d_{2}\right)  \in D(2)}}\,\mathcal{W}%
_{D^{2}}\,
\begin{array}
[c]{c}%
\underrightarrow{\mathcal{W}_{d\in D\mapsto\left(  0,0\right)  \in D^{2}}}\\
\overrightarrow{\mathcal{W}_{d\in D\mapsto\left(  0,d\right)  \in D^{2}}}%
\end{array}
\,\mathcal{W}_{D}%
\]
is a limit diagram in the category $\mathbf{Weil}_{k}$.
\end{lemma}

\begin{theorem}
\label{t3.7}The $\mathbb{R}_{M}$-module
\[
\tau_{\mathcal{W}_{D}}\left(  M\right)  :M\otimes\mathcal{W}_{D}\rightarrow M
\]
is Euclidean with respect to the DG-category $\left(  \mathcal{K}%
/M,\mathbb{R}_{M},\mathbf{T}_{M},\alpha^{M}\right)  $.
\end{theorem}

\begin{proof}
By Lemma \ref{t3.6}\ we have the limit diagram
\begin{align*}
& M\otimes\mathcal{W}_{D(2)}\,\underrightarrow{\mathrm{id}_{M}\otimes
\mathcal{W}_{\left(  d_{1},d_{2}\right)  \in D^{2}\mapsto\left(  d_{1}%
,d_{1}d_{2}\right)  \in D(2)}}\,\\
& M\otimes\mathcal{W}_{D^{2}}\,%
\begin{array}
[c]{c}%
\underrightarrow{\mathrm{id}_{M}\otimes\mathcal{W}_{d\in D\mapsto\left(
0,0\right)  \in D^{2}}}\\
\overrightarrow{\mathrm{id}_{M}\otimes\mathcal{W}_{d\in D\mapsto\left(
0,d\right)  \in D^{2}}}%
\end{array}
\,M\otimes\mathcal{W}_{D}%
\end{align*}
Therefore we have
\[
M\otimes\mathcal{W}_{D(2)}=\overline{\mathbf{T}}_{M}^{\mathcal{W}_{D}}\left(
\tau_{\mathcal{W}_{D}}\left(  M\right)  \right)
\]
while we have
\[
M\otimes\mathcal{W}_{D(2)}=\left(  M\otimes\mathcal{W}_{D}\right)  \times
_{M}\left(  M\otimes\mathcal{W}_{D}\right)
\]
Therefore the desired conclusion follows.
\end{proof}

\section{\label{s4}Differential Forms}

Let $\mathbb{E}$ be a Euclidean $\mathbb{R}$-module which is microlinear and
Weil exponentiable.

\begin{definition}
(\underline{Differential Forms with values in $\mathbb{E}$}) We denote by
$\Omega^{n}(M;\mathbb{E})$ the intersection of all the following equalizers:

\begin{enumerate}
\item the equalizer of the exponential transpose
\[
\mathbb{E}^{M\otimes\mathcal{W}_{D^{n}}}\rightarrow\mathbb{E}^{\mathbb{R}%
\times\left(  M\otimes\mathcal{W}_{D^{n}}\right)  }%
\]
of the composition of
\[
\mathbb{E}^{M\otimes\mathcal{W}_{D^{n}}}\times\left(  \mathbb{R}\times\left(
M\otimes\mathcal{W}_{D^{n}}\right)  \right)  \,\underrightarrow{\mathrm{id}%
_{\mathbb{E}^{M\otimes\mathcal{W}_{D^{n}}}}\times\left(  \underset{i}{\cdot
}\right)  _{M\otimes\mathcal{W}_{D^{n}}}^{\mathbb{R}}}\,\mathbb{E}%
^{M\otimes\mathcal{W}_{D^{n}}}\times\left(  M\otimes\mathcal{W}_{D^{n}%
}\right)
\]
and
\[
\mathbb{E}^{M\otimes\mathcal{W}_{D^{n}}}\times\left(  M\otimes\mathcal{W}%
_{D^{n}}\right)  \,\underrightarrow{\mathrm{ev}}\,\mathbb{E}%
\]
in succession and the exponential transpose
\[
\mathbb{E}^{M\otimes\mathcal{W}_{D^{n}}}\rightarrow\mathbb{E}^{\mathbb{R}%
\times\left(  M\otimes\mathcal{W}_{D^{n}}\right)  }%
\]
of the composition of
\begin{align*}
& \mathbb{E}^{M\otimes\mathcal{W}_{D^{n}}}\times\left(  \mathbb{R}%
\times\left(  M\otimes\mathcal{W}_{D^{n}}\right)  \right) \\
& =\mathbb{R}\times\left(  \mathbb{E}^{M\otimes\mathcal{W}_{D^{n}}}%
\times\left(  M\otimes\mathcal{W}_{D^{n}}\right)  \right)  \,\underrightarrow
{\mathrm{id}_{\mathbb{R}}\times\mathrm{ev}}\,\mathbb{R}\times\mathbb{E}%
\end{align*}
and the scalar multiplication
\[
\mathbb{R}\times\mathbb{E\,\,}\underrightarrow{\mathbb{\cdot}_{\mathbb{E}%
}^{\mathbb{R}}}\,\mathbb{E}%
\]
in succession, where $i$\ ranges over the natural numbers from $1$\ to $n$,
and the exponential transpose of
\[
\left(  \underset{i}{\cdot}\right)  _{M\otimes\mathcal{W}_{D^{n}}}%
^{\mathbb{R}}:\mathbb{R}\times\left(  M\otimes\mathcal{W}_{D^{n}}\right)
\rightarrow M\otimes\mathcal{W}_{D^{n}}%
\]
is
\begin{align*}
& \alpha_{\mathcal{W}_{\left(  a,d_{1},...,d_{n}\right)  \in\mathbb{R}\times
D^{n}\rightarrow\left(  d_{1},...,ad_{i},...,d_{n}\right)  \in D^{n}}}\left(
M\right) \\
& :M\otimes\mathcal{W}_{D^{n}}\rightarrow\left(  M\otimes\mathcal{W}_{D^{n}%
}\right)  ^{\mathbb{R}}=\left(  M\otimes\mathcal{W}_{D^{n}}\right)
\otimes\mathcal{W}_{\mathbb{R}}\\
& =M\otimes\mathcal{W}_{D^{n}\times\mathbb{R}}=M\otimes\mathcal{W}%
_{\mathbb{R}\times D^{n}}%
\end{align*}

\item the equalizer of the exponential transpose
\[
\mathbb{E}^{M\otimes\mathcal{W}_{D^{n}}}\rightarrow\mathbb{E}^{M\otimes
\mathcal{W}_{D^{n}}}%
\]
of the composition of
\[
\mathbb{E}^{M\otimes\mathcal{W}_{D^{n}}}\times\left(  M\otimes\mathcal{W}%
_{D^{n}}\right)  \,\underrightarrow{\mathrm{id}_{\mathbb{E}^{M\otimes
\mathcal{W}_{D^{n}}}}\times\left(  \cdot^{\sigma}\right)  _{M\otimes
\mathcal{W}_{D^{n}}}}\,\mathbb{E}^{M\otimes\mathcal{W}_{D^{n}}}\times\left(
M\otimes\mathcal{W}_{D^{n}}\right)
\]
and
\[
\mathbb{E}^{M\otimes\mathcal{W}_{D^{n}}}\times\left(  M\otimes\mathcal{W}%
_{D^{n}}\right)  \,\underrightarrow{\mathrm{ev}}\,\mathbb{E}%
\]
in succession and the exponential transpose
\[
\mathbb{E}^{M\otimes\mathcal{W}_{D^{n}}}\rightarrow\mathbb{E}^{M\otimes
\mathcal{W}_{D^{n}}}%
\]
of the composition of
\[
\mathbb{E}^{M\otimes\mathcal{W}_{D^{n}}}\times\left(  M\otimes\mathcal{W}%
_{D^{n}}\right)  \,\underrightarrow{\mathrm{ev}}\,\mathbb{E}%
\]
and
\[
\mathbb{E\,}\underrightarrow{\left(  \epsilon_{\sigma}\right)
\mathbb{_{\mathbb{E}}}}\,\mathbb{E}%
\]
in succession, where $\sigma$\ ranges over all the permutation of the set
$\left\{  1,...,n\right\}  $, $\epsilon_{\sigma}$ is the signature of $\sigma
$, $\left(  \epsilon_{\sigma}\right)  \mathbb{_{\mathbb{E}}}$ is the scalar
multiplication by $\epsilon_{\sigma}$, and
\[
\left(  \cdot^{\sigma}\right)  _{M\otimes\mathcal{W}_{D^{n}}}:M\otimes
\mathcal{W}_{D^{n}}\rightarrow M\otimes\mathcal{W}_{D^{n}}%
\]
is
\[
\alpha_{\mathcal{W}_{\left(  d_{1},...,d_{n}\right)  \in D^{n}\rightarrow
\left(  d_{\sigma\left(  1\right)  },...,d_{\sigma\left(  n\right)  }\right)
\in D^{n}}}\left(  M\right)
\]

\end{enumerate}
\end{definition}

\begin{definition}
(\underline{Infinitesimal Integration of Differential Forms}) We define a
morphism
\[
\int_{M,\mathbb{E}}^{n}:\left(  M\otimes\mathcal{W}_{D^{n}}\right)
\times\Omega^{n}(M;\mathbb{E})\rightarrow\mathbb{E}%
\]
in $\mathcal{K}$ to be the composition of
\[
\mathrm{id}_{M\otimes\mathcal{W}_{D^{n}}}\times i_{\Omega^{n}(M;\mathbb{E}%
)\rightarrow\mathbb{E}^{M\otimes\mathcal{W}_{D^{n}}}}:\left(  M\otimes
\mathcal{W}_{D^{n}}\right)  \times\Omega^{n}(M;\mathbb{E})\rightarrow\left(
M\otimes\mathcal{W}_{D^{n}}\right)  \times\mathbb{E}^{M\otimes\mathcal{W}%
_{D^{n}}}%
\]
and
\[
\mathrm{ev}:\left(  M\otimes\mathcal{W}_{D^{n}}\right)  \times\mathbb{E}%
^{M\otimes\mathcal{W}_{D^{n}}}\rightarrow\mathbb{E}%
\]
in succession, where
\[
i_{\Omega^{n}(M;\mathbb{E})\rightarrow\mathbb{E}^{M\otimes\mathcal{W}_{D^{n}}%
}}:\Omega^{n}(M;\mathbb{E})\rightarrow\mathbb{E}^{M\otimes\mathcal{W}_{D^{n}}}%
\]
is the canonical injection.
\end{definition}

\begin{remark}
We should point out that the orthodox definition of the infinitesimal
integration of differential forms in synthetic differential geometry, such as
seen in Chapter 4 of \cite{lav}, is unnecessarily decorated with redundant
fringes. Therein, it is defined as a mapping
\[
\left(  M\otimes\mathcal{W}_{D^{n}}\right)  \times\Omega^{n}(M;\mathbb{E}%
)\rightarrow\mathbb{E}\otimes\mathcal{W}_{D^{n}}%
\]
where the mapping factors through the canonical injection
\[
\hom\left(  \mathbb{E}\otimes\mathcal{W}_{D^{n}}\right)  \rightarrow
\mathbb{E}\otimes\mathcal{W}_{D^{n}}%
\]
with $\hom\left(  \mathbb{E}\otimes\mathcal{W}_{D^{n}}\right)  $ being the
homogeneous subobject of $\mathbb{E}\otimes\mathcal{W}_{D^{n}}$. Since
$\hom\left(  \mathbb{E}\otimes\mathcal{W}_{D^{n}}\right)  $ is canonically
isomorphic to $\mathbb{E}$, such an unnecessarily decoration is to be averted.
\end{remark}

It is trivial to see that

\begin{proposition}
\label{t4.1}The morphism
\[
\int_{M,\mathbb{E}}^{n}:\left(  M\otimes\mathcal{W}_{D^{n}}\right)
\times\Omega^{n}(M;\mathbb{E})\rightarrow\mathbb{E}%
\]
satisfies the following properties:

\begin{enumerate}
\item The composition of morphisms
\[
\left(  \underset{i}{\cdot}\right)  _{M\otimes\mathcal{W}_{D^{n}}}%
^{\mathbb{R}}\times\mathrm{id}_{\Omega^{n}(M;\mathbb{E})}:\mathbb{R}%
\times\left(  M\otimes\mathcal{W}_{D^{n}}\right)  \times\Omega^{n}%
(M;\mathbb{E})\rightarrow\left(  M\otimes\mathcal{W}_{D^{n}}\right)
\times\Omega^{n}(M;\mathbb{E})
\]
and
\[
\int_{M,\mathbb{E}}^{n}:\left(  M\otimes\mathcal{W}_{D^{n}}\right)
\times\Omega^{n}(M;\mathbb{E})\rightarrow\mathbb{E}%
\]
in succession is equal to the composition of morphisms
\[
\mathrm{id}_{\mathbb{R}}\times\int_{M,\mathbb{E}}^{n}:\mathbb{R}\times\left(
M\otimes\mathcal{W}_{D^{n}}\right)  \times\Omega^{n}(M;\mathbb{E}%
)\rightarrow\mathbb{R}\times\mathbb{E}%
\]
and
\[
\mathbb{\cdot}_{\mathbb{E}}^{\mathbb{R}}:\mathbb{R}\times\mathbb{E}%
\,\rightarrow\mathbb{E}%
\]
in succession.

\item The composition of morphisms
\[
\left(  \cdot^{\sigma}\right)  _{M\otimes\mathcal{W}_{D^{n}}}\times
\mathrm{id}_{\Omega^{n}(M;\mathbb{E})}:\left(  M\otimes\mathcal{W}_{D^{n}%
}\right)  \times\Omega^{n}(M;\mathbb{E})\rightarrow\left(  M\otimes
\mathcal{W}_{D^{n}}\right)  \times\Omega^{n}(M;\mathbb{E})
\]
and
\[
\int_{M,\mathbb{E}}^{n}:\left(  M\otimes\mathcal{W}_{D^{n}}\right)
\times\Omega^{n}(M;\mathbb{E})\rightarrow\mathbb{E}%
\]
in succession is equal to the composition of morphisms
\[
\int_{M,\mathbb{E}}^{n}:\left(  M\otimes\mathcal{W}_{D^{n}}\right)
\times\Omega^{n}(M;\mathbb{E})\rightarrow\mathbb{E}%
\]
and
\[
\mathbb{E}\,\underrightarrow{\left(  \epsilon_{\sigma}\right)  _{\mathbb{E}%
\otimes\mathcal{W}_{D^{n}}}}\,\mathbb{E}%
\]
in succession.
\end{enumerate}
\end{proposition}

As should have been expected, we have

\begin{theorem}
\label{t4.2}(\underline{The Fundamental Theorem on Differential Forms}) Given
an object $X$\ in $\mathcal{K}$ and a morphism
\begin{equation}
\varphi:\left(  M\otimes\mathcal{W}_{D^{n}}\right)  \times X\rightarrow
\mathbb{E}\otimes\mathcal{W}_{D^{n}}\text{,}\label{4.2.1}%
\end{equation}
if the morphism (\ref{4.2.1}) satisfies the two conditions in Proposition
\ref{t4.1} with $\Omega^{n}(M;\mathbb{E})$ replaced by $X$ and $\int
_{M,\mathbb{E}}^{n}$\ replaced by $\varphi$, then there exists a unique
morphism
\[
\widehat{\varphi}:X\rightarrow\Omega^{n}(M;\mathbb{E})
\]
such that $\varphi$\ is equal to the composition of morphisms
\[
\mathrm{id}_{M\otimes\mathcal{W}_{D^{n}}}\times\widehat{\varphi}:\left(
M\otimes\mathcal{W}_{D^{n}}\right)  \times X\,\rightarrow\left(
M\otimes\mathcal{W}_{D^{n}}\right)  \times\Omega^{n}(M;\mathbb{E})
\]
and
\[
\int_{M,\mathbb{E}}^{n}:\left(  M\otimes\mathcal{W}_{D^{n}}\right)
\times\Omega^{n}(M;\mathbb{E})\rightarrow\mathbb{E}%
\]
in succession.
\end{theorem}

\begin{proof}
The theorem follows rather directly from the universal construction of
$\Omega^{n}(M;\mathbb{E})$. Take the exponential transpose
\[
\widetilde{\varphi}:X\rightarrow\left(  \mathbb{E}\otimes W_{D^{n}}\right)
^{M\otimes W_{D^{n}}}%
\]
of (\ref{4.2.1}), which factors, by the two conditions on $\varphi$, into a
morphism
\[
\widehat{\varphi}:X\rightarrow\Omega^{n}(M;\mathbb{E})
\]
followed by the canonical monomorphism
\[
\Omega^{n}(M;\mathbb{E})\rightarrow\mathbb{E}^{M\otimes W_{D^{n}}}%
\]
It is not difficult to see that the above $\widehat{\varphi}$\ is the desired
unique morphism in the theorem. The details can safely be left to the reader.
\end{proof}

\section{\label{s5}The Exterior Differentiation}

\begin{definition}
Given natural numbers $n,i$ with $1\leq i\leq n+1$, we define a morphism
\[
\left(  \int_{M,\mathbb{E}}^{n}\right)  _{i}:\left(  M\otimes\mathcal{W}%
_{D^{n+1}}\right)  \times\Omega^{n}(M;\mathbb{E})\rightarrow\mathbb{E}%
\]
in $\mathcal{K}$ to be
\begin{align*}
& \left(  M\otimes\mathcal{W}_{D^{n+1}}\right)  \times\Omega^{n}%
(M;\mathbb{E})\\
& \,\underrightarrow{\left(  \mathrm{id}_{M}\otimes\mathcal{W}_{(d_{1}%
,...,d_{n+1})\in D^{n+1}\mapsto(d_{1},...,d_{i-1},d_{i+1},...,d_{n+1}%
,d_{i})\in D^{n+1}}\right)  \times\mathrm{id}_{\Omega^{n}(M;\mathbb{E})}}\\
& \left(  M\otimes\mathcal{W}_{D^{n+1}}\right)  \times\Omega^{n}%
(M;\mathbb{E})\\
& =\left(  M\otimes\mathcal{W}_{D^{n+1}}\right)  \times\left(  \Omega
^{n}(M;\mathbb{E})\otimes k\right) \\
& \,\underrightarrow{\mathrm{id}_{M\otimes\mathcal{W}_{D^{n+1}}}\times\left(
\mathrm{id}_{\Omega^{n}(M;\mathbb{E})}\otimes\left(  k\rightarrow k[X]/\left(
X^{2}\right)  =\mathcal{W}_{D}\right)  \right)  }\,\\
& \left(  M\otimes\mathcal{W}_{D^{n+1}}\right)  \times\left(  \Omega
^{n}(M;\mathbb{E})\otimes\mathcal{W}_{D}\right) \\
& =\left(  \left(  M\otimes\mathcal{W}_{D^{n}}\right)  \otimes\mathcal{W}%
_{D}\right)  \times\left(  \Omega^{n}(M;\mathbb{E})\otimes\mathcal{W}%
_{D}\right) \\
& =\left(  \left(  M\otimes\mathcal{W}_{D^{n}}\right)  \times\Omega
^{n}(M;\mathbb{E})\right)  \otimes\mathcal{W}_{D}\\
& \,\underrightarrow{\int_{M,\mathbb{E}}^{n}\otimes\mathrm{id}_{\mathcal{W}%
_{D}}}\,\mathbb{E\otimes}\mathcal{W}_{D}\\
& =\mathbb{E\otimes E\rightarrow E}%
\end{align*}
where the last morphism
\[
\mathbb{E\otimes E\rightarrow E}%
\]
is the second projection, and
\[
k\rightarrow k[X]/\left(  X^{2}\right)  =\mathcal{W}_{D}%
\]
is the canonical morphism.
\end{definition}

\begin{notation}
We denote by
\[
\left(  \partial_{i}^{n+1}\right)  _{M\otimes\mathcal{W}_{D^{n+1}}}%
\]
the morphism
\[
\mathrm{id}_{M}\otimes\mathcal{W}_{(d_{1},...,d_{n+1})\in D^{n+1}\mapsto
(d_{1},...,d_{i-1},d_{i+1},...,d_{n+1},d_{i})\in D^{n+1}}%
\]

\end{notation}

In order to establish the fundamental theorem on exterior differentiation, we
need two lemmas, which go as follows:

\begin{lemma}
\label{t5.1}The composition of morphisms
\[
\left(  \underset{j}{\cdot}\right)  _{M\otimes\mathcal{W}_{D^{n}}}%
^{\mathbb{R}}\times\mathrm{id}_{\Omega^{n}(M;\mathbb{E})}:\mathbb{R}%
\times\left(  M\otimes\mathcal{W}_{D^{n+1}}\right)  \times\Omega
^{n}(M;\mathbb{E})\rightarrow\left(  M\otimes\mathcal{W}_{D^{n+1}}\right)
\times\Omega^{n}(M;\mathbb{E})
\]
and
\[
\left(  \int_{M,\mathbb{E}}^{n}\right)  _{i}:\left(  M\otimes\mathcal{W}%
_{D^{n+1}}\right)  \times\Omega^{n}(M;\mathbb{E})\rightarrow\mathbb{E}%
\]
in succession is equal to the composition of morphisms
\[
\mathrm{id}_{\mathbb{R}}\times\left(  \int_{M,\mathbb{E}}^{n}\right)
_{i}:\mathbb{R}\times\left(  M\otimes\mathcal{W}_{D^{n+1}}\right)
\times\Omega^{n}(M;\mathbb{E})\rightarrow\mathbb{R}\times\mathbb{E}%
\]
and
\[
\mathbb{\cdot}_{\mathbb{E}}^{\mathbb{R}}:\mathbb{R}\times\mathbb{E}%
\,\rightarrow\mathbb{E}%
\]
in succession.
\end{lemma}

\begin{proof}
For $j<i$, it is easy to see that
\begin{align*}
&  \left(  \partial_{i}^{n+1}\right)  _{M\otimes\mathcal{W}_{D^{n+1}}}%
\circ\left(  \underset{j}{\cdot}\right)  _{M\otimes\mathcal{W}_{D^{n+1}}%
}^{\mathbb{R}}\\
&  =\left(  \left(  \underset{j}{\cdot}\right)  _{M\otimes\mathcal{W}_{D^{n}}%
}^{\mathbb{R}}\otimes\mathrm{id}_{W_{D}}\right)  \circ\left(  \mathrm{id}%
_{\mathbb{R}}\times\left(  \partial_{i}^{n+1}\right)  _{M\otimes
\mathcal{W}_{D^{n+1}}}\right)
\end{align*}
while, for $j>i$, it is also easy to see that
\begin{align*}
&  \left(  \partial_{i}^{n+1}\right)  _{M\otimes\mathcal{W}_{D^{n+1}}}%
\circ\left(  \underset{j}{\cdot}\right)  _{M\otimes\mathcal{W}_{D^{n+1}}%
}^{\mathbb{R}}\\
&  =\left(  \left(  \underset{j-1}{\cdot}\right)  _{M\otimes\mathcal{W}%
_{D^{n}}}^{\mathbb{R}}\otimes\mathrm{id}_{W_{D}}\right)  \circ\left(
\mathrm{id}_{\mathbb{R}}\times\left(  \partial_{i}^{n+1}\right)
_{M\otimes\mathcal{W}_{D^{n+1}}}\right)
\end{align*}
Therefore, for $j\neq i$, that
\begin{align*}
&  \left(  \int_{M,\mathbb{E}}^{n}\right)  _{i}\circ\left(  \left(
\underset{j}{\cdot}\right)  _{M\otimes\mathcal{W}_{D^{n+1}}}^{\mathbb{R}%
}\times\mathrm{id}_{\Omega^{n}(M;\mathbb{E})}\right) \\
&  =\left(  \mathbb{\cdot}_{\mathbb{E}}^{\mathbb{R}}\right)  \circ\left(
\mathrm{id}_{\mathbb{R}}\times\left(  \int_{M,\mathbb{E}}^{n}\right)
_{i}\right)
\end{align*}
follows directly. It remains to show that
\begin{align*}
&  \left(  \int_{M,\mathbb{E}}^{n}\right)  _{i}\circ\left(  \left(
\underset{i}{\cdot}\right)  _{M\otimes\mathcal{W}_{D^{n+1}}}^{\mathbb{R}%
}\times\mathrm{id}_{\Omega^{n}(M;\mathbb{E})}\right) \\
&  =\left(  \mathbb{\cdot}_{\mathbb{E}}^{\mathbb{R}}\right)  \circ\left(
\mathrm{id}_{\mathbb{R}}\times\left(  \int_{M,\mathbb{E}}^{n}\right)
_{i}\right)
\end{align*}
which follows readily from
\begin{align*}
&  \left(  \partial_{i}^{n+1}\right)  _{M\otimes\mathcal{W}_{D^{n+1}}}%
\circ\left(  \underset{i}{\cdot}\right)  _{M\otimes\mathcal{W}_{D^{n+1}}%
}^{\mathbb{R}}\\
&  =\left(  \underset{n+1}{\cdot}\right)  _{M\otimes\mathcal{W}_{D^{n+1}}%
}^{\mathbb{R}}\circ\left(  \partial_{i}^{n+1}\right)  _{M\otimes
\mathcal{W}_{D^{n+1}}}%
\end{align*}

\end{proof}

\begin{lemma}
\label{t5.2}Given a permutation $\sigma$ of $\left\{  1,...,n+1\right\}  $, we
have
\begin{align*}
&  \left(  \sum_{i=1}^{n+1}(-1)^{i+1}\left(  \int_{M,\mathbb{E}}^{n}\right)
_{i}\right)  \circ\left(  \left(  \cdot^{\sigma}\right)  _{M\otimes
\mathcal{W}_{D^{n+1}}}\times\mathrm{id}_{\Omega^{n}(M;\mathbb{E})}\right) \\
&  =\varepsilon_{\sigma}\sum_{i=1}^{n+1}(-1)^{i+1}\left(  \int_{M,\mathbb{E}%
}^{n}\right)  _{i}%
\end{align*}

\end{lemma}

\begin{proof}
We notice that
\begin{align*}
&  \left(  \partial_{i}^{n+1}\right)  _{M\otimes\mathcal{W}_{D^{n+1}}}%
\circ\left(  \cdot^{\sigma}\right)  _{M\otimes\mathcal{W}_{D^{n+1}}}\\
&  =\left(  \mathrm{id}_{M}\otimes W_{\left(  \delta_{i}^{\sigma}\right)
_{D^{n}}}\right)  \left(  \cdot^{\delta_{i}^{\sigma}}\right)  _{M\otimes
\mathcal{W}_{D^{n+1}}}\circ\left(  \partial_{\sigma^{-1}\left(  i\right)
}^{n+1}\right)  _{M\otimes\mathcal{W}_{D^{n+1}}}%
\end{align*}
where $\delta_{i}^{\sigma}$ is the permutation of $\left\{  1,...,n\right\}  $
with
\begin{align*}
\delta_{i}^{\sigma}\left(  j\right)   & =\sigma\left(  j\right)  \text{ in
case of }j<\sigma^{-1}(i)\text{ and }\sigma\left(  j\right)  <i\text{;}\\
\delta_{i}^{\sigma}\left(  j\right)   & =\sigma\left(  j+1\right)  \text{ in
case of }j\geqq\sigma^{-1}(i)\text{ and }\sigma\left(  j\right)  <i\text{;}\\
\delta_{i}^{\sigma}\left(  j\right)   & =\sigma\left(  j\right)  -1\text{ in
case of }j<\sigma^{-1}(i)\text{ and }\sigma\left(  j\right)  \geqq i\text{;}\\
\delta_{i}^{\sigma}\left(  j\right)   & =\sigma\left(  j+1\right)  -1\text{ in
case of }j\geqq\sigma^{-1}(i)\text{ and }\sigma\left(  j\right)  \geqq
i\text{.}%
\end{align*}
We notice also that
\begin{align*}
&  \left(  \int_{M,\mathbb{E}}^{n}\right)  \circ\left(  \left(  \cdot
^{\delta_{i}^{\sigma}}\right)  _{M\otimes\mathcal{W}_{D^{n}}}\times
\mathrm{id}_{\Omega^{n}(M;\mathbb{E})}\right) \\
&  =\varepsilon_{\delta_{i}^{\sigma}}\int_{M,\mathbb{E}}^{n}%
\end{align*}
and
\[
\varepsilon_{\delta_{i}^{\sigma}}=(-1)^{\sigma^{-1}(i)-i}\varepsilon_{\sigma}%
\]
Therefore the desired statement follows.
\end{proof}

\begin{theorem}
\label{t5.3}(\underline{The Fundamental Theorem on Exterior Differentiation})
There exists a unqiue morphism
\[
\mathbf{d}_{n}:\Omega^{n}(M;\mathbb{E})\rightarrow\Omega^{n+1}(M;\mathbb{E})
\]
in $\mathcal{K}$ such that the the composition of morphisms
\[
\mathrm{id}_{M\otimes\mathcal{W}_{D^{n+1}}}\times\mathbf{d}_{n}:\left(
M\otimes\mathcal{W}_{D^{n+1}}\right)  \times\Omega^{n}(M;\mathbb{E}%
)\rightarrow\left(  M\otimes\mathcal{W}_{D^{n+1}}\right)  \times\Omega
^{n+1}(M;\mathbb{E})
\]
and
\[
\int_{M,\mathbb{E}}^{n+1}:\left(  M\otimes\mathcal{W}_{D^{n+1}}\right)
\times\Omega^{n+1}(M;\mathbb{E})\rightarrow\mathbb{E}%
\]
is equal to the morphism
\[
\sum_{i=1}^{n+1}(-1)^{i+1}\left(  \int_{M,\mathbb{E}}^{n}\right)  _{i}:\left(
M\otimes\mathcal{W}_{D^{n+1}}\right)  \times\Omega^{n}(M;\mathbb{E}%
)\rightarrow\mathbb{E}%
\]

\end{theorem}

\begin{proof}
This follows easily from Lemmas \ref{t5.1} and \ref{t5.2} and Theorem
\ref{t4.2}.
\end{proof}

\end{document}